\documentclass[12pt]{amsart}

\usepackage{amssymb,amscd,amsthm, verbatim,amsmath,color,fancyhdr, mathrsfs}
\usepackage{graphicx}
\usepackage{turnstile}
\usepackage{slashed}

\theoremstyle{definition} 
\newtheorem{thm}{Theorem}[section]
\newtheorem{prop}[thm]{Proposition}

\newtheorem{rem}[thm]{Remark}

\newtheorem{ex}[thm]{Example}
\newtheorem{defn}[thm]{Definition}

\title{Systolic Inequality and Scalar Curvature}

\author{Shunichiro Orikasa}
\address{Department of Mathematics, Graduate School of Science, Kyoto University, Sakyoku, Kyoto 606–8502, Japan}
\email{orikasa.shunichiro.34x@st.kyoto-u.ac.jp}
\date{\today}

\begin{document}

\begin{abstract}
We investigate the interaction between systolic geometry and positive scalar curvature through spinorial methods. Our main theorem establishes an upper bound for the two-dimensional stable systole on certain high-dimensional manifolds with positive scalar curvature under a suitable stretch-scale condition. The proof combines techniques from geometric measure theory, reminiscent of Gromov’s systolic inequality, with curvature estimates derived from the Gromov–Lawson relative index theorem. This approach provides a new framework for studying the relationship between positive scalar curvature metrics and systolic geometry in higher-dimensional manifolds.
\end{abstract}
\maketitle

\section{Introduction}

The study of systolic inequalities has evolved into a rich interplay between Riemannian geometry, topology, and geometric analysis. Originally introduced by Loewner and Pu for the torus and real projective plane, these inequalities relate the volume of a manifold to the minimal size of non-trivial cycles in its homology. Gromov’s foundational work \cite{gromov1992systoles} extended these ideas to higher dimensions and broader classes of manifolds, revealing deep geometric and topological obstructions to metric degenerations. A comprehensive survey can be found in Guth \cite{guth2010metaphors}. A particularly subtle phenomenon arises when one considers \emph{systolic freedom}---the failure of any universal inequality between volume and systoles. For instance, in the integral setting, Gromov \cite{gromov1992systoles} have constructed manifolds of bounded volume where certain homological systoles diverge. However, when coefficients are taken in $\mathbb{Z}/2$, the situation is significantly more delicate. Freedman \cite{freedman1999z2}, \cite{freedman2002z2} first exhibited $\mathbb{Z}/2$-systolic freedom, but only at polylogarithmic scales. Alpert, Balitskiy, and Guth \cite{alpert2024systolic} recently established that for any closed $n$-dimensional Riemannian manifold with bounded local geometry and nontrivial $\mathbb{Z}/2$-homology, the product $\mathrm{sys}_1(M;\mathbb{Z}/2) \cdot \mathrm{sys}_{n-1}(M;\mathbb{Z}/2)$ is \emph{almost rigid}: it cannot grow faster than 
$\mathrm{Vol}(M)^{1+\varepsilon}$ for arbitrarily small $\varepsilon > 0$. 

On the other hand, the geometric setting of positive scalar curvature (PSC) metrics offers another constraint on systolic phenomena. Richard \cite{richard20202} recently showed that positive scalar curvature metrics on $S^2\times S^2$ with large left stretch cannot have large left 2-systole and Zhu treated the case of 
$S^2\times T^n$ for $(n + 2 \leq 7)$ in \cite{zhu2020rigidity}.

In the three-dimensional setting, Stern \cite{stern2022scalar} introduced a new tool for relating scalar curvature to the global structure of closed 3-manifolds, based on the study
of harmonic one-forms with integral periods. It leads to the extension of a theorem of Kronheimer and Mrowka \cite{kronheimer1997scalar} characterizing the Thurston norm using the Seiberg–Witten equations and
the Weitzenb\"ock formula for the Dirac operator.

In this paper, we aim to bridge these perspectives by using spinor methods developed by Gromov-Lawson \cite{gromov1983positive}. The Gromov–Lawson relative index theorem says that if two complete Riemannian manifolds are identical outside compact sets, then the difference of their Dirac indices depends only on the geometry inside those compact sets. In other words, the index is a relative invariant determined locally by how the metrics differ, not by the infinite ends where they agree. 

In order to state our main theorem, we briefly introduce some terminology. 
A \emph{scalably contractible complex} is a partial generalization of enlargeable manifolds, 
characterized by the existence of arbitrarily contracting maps to spheres at large scales. 
For such a complex $K$, the \emph{distance sphere} $S_K$ denotes the boundary of a small tubular 
neighborhood around $K$. 
The metric condition in the theorem is formulated in terms of an \emph{adaptable metric}, 
meaning that the scalar curvature lower bound extends to the end-completed space. 
Finally, the invariant $L_{n-k}(X)$, called the $(n-k)$-\emph{stretch scale}, 
measures the minimal Lipschitz constant of maps from $X$ to spheres, 
and plays a role analogous to the hyperspherical radius.
The theorem below is the main result of the paper.
\begin{thm}
Let \( K \) be a scalably contractible complex of dimension \( k \), and let \( S_K \) denote a distance sphere around \( K \) of dimension \( n-2 \). Consider the trivial disk bundle \( X = S_K \times D^2 \), equipped with a Riemannian metric \( g \). Suppose $n$ is even and $g$ is adaptable.

One of the following statements holds:
\begin{itemize}
    \item the scalar curvature of \( X \) satisfies \( Sc(X) \leq C_{1,n}\cdot\sigma_X \),
    \item The stable relative 2-systole of \( X \) admits the upper bound
\[
\mathrm{stsys}_2(X, \partial X) \leq \frac{C_{2,n}}{C_{1,n}\cdot\sigma_X},
\]
\end{itemize}

where \( C_{1,n}, C_{2,n} \) is a constant depending only on the dimension \( n \), and
\[
\sigma_X = L_{n-k}(X)^{2}
\]
with \( L_{n-k}(X) \) denoting the \((n-k)\)-stretch scale of \( X \).
\end{thm}
This method opens up a new approach for studying the relationship between positive scalar curvature metrics and systolic geometry in higher-dimensional manifolds. Moreover, it is based on estimates involving the dual of the curvature 2-form of the spinor bundle, which reveals an intriguing similarity to the techniques of Kronheimer–Mrowka \cite{kronheimer1997scalar}. Our methods combine geometric measure theory which is an analogous method used in the proof of Gromov's systolic inequality for complex projective space $\mathbb{C}P^n$ and techniques from scalar curvature comparison using spinor methods.
\subsection*{Acknowledgment.}
The author would like to thank his advisor Professor Tsuyoshi Kato
for his constant encouragement and helpful suggestions. The authors wish to thank Demetre Kazaras for stimulating conversations related to this
work. In addition, he acknowledges the Joint Japan/US Collaborative Workshop on Geometric Analysis II for providing an inspiring environment and fruitful discussions that contributed to the development of this work.

\section{Preliminaries}
In this section, we briefly review some basic notions of systolic geometry and some results from geometric measure theory. The reader is referred to \cite{katz2007systolic} by Katz for a more details on systole.

Let \( M \) be a smooth, oriented closed Riemannian manifold. The \emph{comass norm} \( \|\omega\|^* \) of a differential \( k \)-form \( \omega \in \Omega^k(M) \) is defined as
\[
\|\omega\|^* := \sup \left\{ \omega(v_1, \dots, v_k) \;\middle|\; v_i \in T_pM,\ \|v_i\| = 1,\ p \in M \right\},
\]
i.e., the supremum of the value of \( \omega \) on orthonormal \( k \)-tuples.

Dually, for a \( k \)-dimensional singular chain \( c \), or more generally for a \( k \)-dimensional current \( T \), the \emph{mass norm} is defined by
\[
\|c\| := \sup \left\{ \langle \omega, c \rangle \;\middle|\; \omega \in \Omega^k(M),\ \|\omega\|_* \leq 1 \right\}.
\]

The mass norm on \( H_k(M;\mathbb{R}) \) is defined via
\[
\|\alpha\| := \inf \left\{ \|c\| \;\middle|\; [c] = \alpha \right\},
\]
where \( \|c\| \) denotes the mass of a singular cycle \( c \).

The comass norm on \( H^k(M;\mathbb{R}) \) is defined by
\[
\|\omega\|^* := \sup \left\{ \omega(v_1, \dots, v_k) \;\middle|\; v_i \in T_pM,\ \|v_i\| = 1,\ p \in M \right\}.
\]

This defines a norm on the homology \( H_k(M;\mathbb{R}) \), dual to the comass norm on cohomology \( H^k(M;\mathbb{R}) \) via the Kronecker pairing \cite{federer1974real}. 
Let \( L \subset \mathbb{R}^n \) be a full-rank lattice.
Define the shortest vector length by
\[
\lambda_1(L) := \min \{ \|v\| \mid v \in L \setminus \{0\} \}.
\]

Historically, the first lower bound on the volume of a Riemannian manifold in terms of its systole is due to Loewner.  

\begin{thm}[Loewner]
    For every Riemannian metric $g$ on the torus $T^2$, one has
    \[
        \operatorname{sys}_1(g)^2 \leq \gamma_2 \, \operatorname{area}(g).
    \]
\end{thm}

We begin by recalling the definition of the homology $k$-systole. For a Riemannian manifold $(M,g)$, the homology $k$-systole $\operatorname{sysh}_k(g)$ is the minimum $k$-volume among all nontrivial integer homology classes:
\[
    \operatorname{sysh}_k(g) = \inf\{ \operatorname{vol}_k(h) \; ; \; h \in H_k(M;\mathbb{Z}) \setminus \{0\} \}.
\]
In particular, $\operatorname{sysh}_{n-1}(g)$ is the infimum of the $(n-1)$-dimensional volumes of compact, orientable, non-separating hypersurfaces in $M$.

We denote by $H_k(M;\mathbb{Z})_{\mathbb{R}}$ the image of $H_k(M;\mathbb{Z})$ in $H_k(M;\mathbb{R})$, and by $h_{\mathbb{R}}$ the image of $h \in H_k(M;\mathbb{Z})$. Recall that $H_k(M;\mathbb{Z})$ forms a lattice in $H_k(M;\mathbb{R})$.

It is clear that
\[
    \|h_{\mathbb{R}}\| \leq \operatorname{vol}_k(h), \quad \forall h \in H_k(M;\mathbb{Z}).
\]

Federer \cite{federer1974real} studied the relation between $\|h_{\mathbb{R}}\|$ and $\operatorname{vol}_k(h)$, proving the following:

\begin{prop}
    For $h \in H_k(M;\mathbb{Z})$ with $1 \leq k < n$, we have
    \[
        \|h_{\mathbb{R}}\| = \lim_{j \to \infty} \, \frac{1}{j} \, \operatorname{vol}_k(jh).
    \]
    Moreover, if $k=n-1$, then
    \[
        \|h_{\mathbb{R}}\| = \operatorname{vol}_{n-1}(h).
    \]
\end{prop}

Passing from integer cycles to real cycles, one obtains the stable norm:

\begin{defn}
    The stable $k$-systole $\operatorname{stsys}_k(g)$ of $(M,g)$ is defined by
    \[
        \operatorname{stsys}_k(g) = \min \bigl\{ \|h\| \; ; \; h \in H_k(M;\mathbb{Z})_{\mathbb{R}} \setminus \{0\} \bigr\} 
        = \lambda_1 \bigl(H_k(M;\mathbb{Z})_{\mathbb{R}}, \|\cdot\| \bigr)\}.
    \]
\end{defn}

In particular, one has $\operatorname{stsys}_{n-1}(g) = \operatorname{sysh}_{n-1}(g)$.

We now state Gromov’s systolic inequality for complex projective space, whose proof relies on the cup product decomposition of the fundamental class.

\begin{thm}[Gromov]
    Every Riemannian metric $g$ on complex projective space $\mathbb{C}P^n$ satisfies the sharp inequality
    \[
        \operatorname{stsys}_2(g)^n \leq n! \, \operatorname{vol}_{2n}(g),
    \]
    with equality attained for the Fubini--Study metric.
\end{thm}

Here, we recall the sketch of proof.
Let $\alpha \in H_2(\mathbb{C}P^n;\mathbb{Z}) \cong \mathbb{Z}$ be a generator, and let $\omega \in H^2(\mathbb{C}P^n;\mathbb{Z}) \cong \mathbb{Z}$ be its dual in cohomology. Then the cup power $\omega^n$ generates $H^{2n}(\mathbb{C}P^n;\mathbb{Z}) \cong \mathbb{Z}$. Choose a differential $2$-form $\eta$ representing $\omega$. Then
\[
    1 \;\leq\; \int_{\mathbb{C}P^n} \|\eta^n\| \, d\mathrm{vol} 
    \;\leq\; n! \, \|\eta\|^n \, \operatorname{vol}_{2n}(\mathbb{C}P^n,g),
\]
where the last inequality follows from the Wirtinger inequality.  
Since $(H_2(M;\mathbb{Z})_{\mathbb{R}}, \|\cdot\|)$ and $(H^2(M;\mathbb{Z})_{\mathbb{R}}, \|\cdot\|^*)$ are dual normed lattices, one has $\|\alpha\|\cdot\|\omega\|^*=1$. This yields the desired result.

\section{Enlargeability and Scalability}
We recall the key concept of \emph{enlargeability}, introduced by Gromov and Lawson \cite{gromov1983positive}. This notion has recently attracted attention in connection with Schoen and Yau’s minimal surface techniques \cite{li2025coveringinstabilityexistencepositive}. In this paper, we define a partial generalization of this concept.
\subsection{Enlargeable Manifolds}
\begin{defn}[$\epsilon$-Contracting Map]
A map $f:X \to Y$ between Riemannian manifolds is \emph{$\epsilon$-contracting} for a given $\varepsilon>0$ if 
\[
\|f_* v\| \le \varepsilon \|v\|
\] 
for all tangent vectors $v$ in $X$. Equivalently, for any piecewise smooth curve $y \subset X$,
\[
\text{length}(f(y)) \le \varepsilon \, \text{length}(y).
\]
\end{defn}
While $\epsilon$-contracting maps are easy to construct in general (e.g., constant maps), if $X$ and $Y$ are compact oriented manifolds of the same dimension, the existence of a non-zero degree $\epsilon$-contracting map $f: X \to Y$ implies that $X$ is ``larger than'' $Y$ on the order of $1/\varepsilon$.

\begin{defn}[Enlargeable Manifold]
A compact oriented Riemannian $n$-manifold $X$ is \emph{enlargeable} if for every $\varepsilon > 0$, there exists an spin Riemannian covering space $\tilde{X}$ that admits an $\epsilon$-contracting map to the standard $n$-sphere $S^n(1)$, which is constant outside a compact set and of non-zero degree.
\end{defn}

\begin{rem}
A map is constant at infinity if it is constant outside some compact set. The degree is defined via integration of an $n$-form with non-zero integral on $S^n$.
\end{rem}

\begin{ex}
The flat torus $T^n = \mathbb{R}^n / \mathbb{Z}^n$ is enlargeable. For each integer $k > 0$, the covering torus $T^n_k = \mathbb{R}^n / (k \mathbb{Z})^n$ admits a $(\pi/k)$-contracting map to $S^n(1)$ of degree 1.
\end{ex}

We summarize here the key properties of enlargeable manifolds. The class of enlargeable manifolds is closed under products, connected sums, and homotopy changes. It contains many important $K(\pi,1)$-manifolds as basic examples \cite{gromov1983positive}, \cite{lawson2016spin}.
\begin{thm}
In the category of compact manifolds, the following hold:
\begin{enumerate}
\item Enlargeability is independent of the Riemannian metric.
\item Enlargeability depends only on the homotopy type.
\item The product of enlargeable manifolds is enlargeable.
\item The connected sum of any manifold with an enlargeable manifold is enlargeable.
\item Any manifold admitting a map of non-zero degree onto an enlargeable manifold is itself enlargeable.
\end{enumerate}
\end{thm}

\subsection{Scalably Contractible Complexes and Review from PL geometry}
In this section, we introduce a concept that partially generalizes enlargeability and discuss its construction and related properties.
\begin{defn}
Let \( X \) be a connected, finite CW complex of dimension \( n \), equipped with a metric \( d \). 
We say that \( X \) is \emph{scalably contractible} if for every \( \varepsilon > 0 \), there exists a finite covering space 
\[
p : \tilde{X}_\varepsilon \to X
\]
and an \( \varepsilon \)-Lipschitz map
\[
f_\varepsilon : \tilde{X}_\varepsilon \to S^n
\]
such that: $\displaystyle\sum_{e^n_\lambda}[f(e^n_\lambda):S^n]$ is non-zero where $e^n_\lambda$ are top-dimensional cells of $X$ with orientation and $[f(e^n_\lambda):S^n]$ is is the attaching coefficient.
\end{defn}
If \( X \) is a closed \( n \)-dimensional Riemannian manifold, this space satisfies the classical definition of enlargeability by Gromov-Lawson.

In order to state the main theorem, we introduce some terminology from PL geometry. In the following, all simplicial complexes are assumed to be finite and embedded in Euclidean space. 

Let \( K \) be a finite abstract simplicial complex of dimension \( n \). Suppose \( |K| \hookrightarrow \mathbb{R}^{N} \) is a topological embedding into some Euclidean space in such a way that:

\begin{itemize}
  \item The \emph{interior} of each simplex \( \sigma^k \in K \) is smoothly embedded,
  \item Simplices intersect only along common faces.
\end{itemize}

For \( \varepsilon > 0 \), the \emph{$\epsilon$-neighborhood} of \( K \) is defined as:
\[
U_\varepsilon(K) := \{ x \in \mathbb{R}^n \mid \operatorname{dist}(x, K) < \varepsilon \},
\]
where \( \operatorname{dist}(x, K) = \inf_{y \in K} \|x - y\| \) is the Euclidean distance.

The \emph{boundary} of the $\epsilon$-neighborhood is defined as:
\[
\partial U_\varepsilon(K) := \overline{U_\varepsilon(K)} \setminus \operatorname{Int}(U_\varepsilon(K)).
\]

If \( K \) is a smooth submanifold (or embedded piecewise linearly), and \( \varepsilon \) is sufficiently small, then \( \partial U_\varepsilon(K) \) is a smooth \((n-1)\)-dimensional hypersurface in \( \mathbb{R}^n \), called the \emph{distance sphere} around \( K \).

Consider a regular neighborhood (also called a \emph{thickening}) \( U \subset \mathbb{R}^{N+1} \) of \( |K| \), which is an \( N \)-dimensional manifold with boundary satisfying:
\[
|K| \subset U, \quad \text{and} \quad U \simeq |K| \text{ (homotopy equivalence)}.
\]

In particular, the above homotopy equivalence is given by a deformation retract $r_K: U\to |K|$. 
Then the boundary \( \partial U \) can be decomposed into pieces corresponding to the simplices of \( K \). Specifically, to each \( k \)-simplex \( \sigma^k \in K \), we assign a block of the form:
\[
S^{N-k} \times D^k,
\]
where:
\begin{itemize}
  \item \( D^k \) corresponds to the simplex \( \sigma^k \),
  \item \( S^{N-k} \) is the normal sphere around \( \sigma^k \) in the ambient space.
\end{itemize}

The boundary \( \partial U \) is obtained by gluing these blocks along their boundaries, according to the face relations among the simplices of \( K \). We call \( \partial U\) as a \emph{distance sphere of $K$} and denote it by $S_K$.
\begin{ex}
Consider a finite graph \( G = (V, E) \). We define a topological space by assigning:

\begin{itemize}
  \item to each vertex \( v \in V \), a copy of the \( n \)-sphere \( S^n \), and
  \item to each edge \( e = \{v, w\} \in E \), a connected sum operation between the spheres \( S^n_v \) and \( S^n_w \) assigned to the endpoints.
\end{itemize}

More precisely, for each edge \( e = \{v, w\} \), we perform the connected sum \( S^n_v \# S^n_w \) by removing a small open ball from each sphere and identifying their boundary spheres via a homeomorphism:
\[
S^n_v \# S^n_w = \left( S^n_v \setminus B^n \right) \cup_{\varphi_e} \left( S^n_w \setminus B^n \right),
\]
where \( \varphi_e : S^{n-1} \to S^{n-1} \) is a gluing map along the boundary spheres.

By performing this operation for every edge in \( G \), we obtain a space \( X( G) \) formed as the connected sum of spheres arranged according to the graph structure of \( G \). (This is called the connected sum along the graph \( G \).) This construction provides an abstract way for obtaining distance spheres of graphs.
\end{ex}
We introduce two notions to handle metrics that can be extended to the ends in a controlled way. The following concept, \emph{stretch scale}, is analogous to the hyperspherical radius.

\begin{defn}[Stretch scale]
Let $M$ be the trivial 2-disk bundle, oriented and endowed with Riemannian metric $g$. The $l$-\emph{stretch scale} of $M$, denoted $L_l(M)$, 
is defined as the smallest Lipschitz constant of a bundle map
\[
  f \colon M \longrightarrow S^{\,l-2} \times D^2
\]
with non-zero degree.
\end{defn}

\begin{defn}[Adaptable metric]
Let $X$ be a smooth manifold with bounbary. A metric $g$ on $X$ is called \emph{adaptable} if there exists 
a complete extension $\bar{g}$
of $g$ over the end-completed space with $Sc(\bar{g})\geq \min Sc(g)$
\[
  X' := X\,\cup\, \bigl(\partial X\times [0,+\infty)\bigr).
\]
\end{defn}

\section{Proof of Main Theorem}
In this section, we prove the main theorem of this paper, which provides an estimate for the stable two-dimensional relative systole on high-dimensional Riemannian manifolds with  positive scalar curvature metrics.
\begin{thm}
Let \( K \) be a scalably contractible complex of dimension \( k \), and let \( S_K \) denote a distance sphere around \( K \) of dimension \( n-2 \). Consider the trivial disk bundle \( X = S_K \times D^2 \), equipped with a Riemannian metric \( g \). Suppose $n$ is even and $g$ is adaptable.

One of the following statements holds:
\begin{itemize}
    \item the scalar curvature of \( X \) satisfies \( Sc(X) \leq C_{1,n}\cdot\sigma_X \),
    \item The stable relative 2-systole of \( X \) admits the upper bound
\[
\mathrm{stsys}_2(X, \partial X) \leq \frac{C_{2,n}}{C_{1,n}\cdot\sigma_X},
\]
\end{itemize}

where \( C_{1,n}, C_{2,n} \) is a constant depending only on the dimension \( n \), and
\[
\sigma_X = L_{n-k}(X)^{2}
\]
with \( L_{n-k}(X) \) denoting the \((n-k)\)-stretch scale of \( X \).
\end{thm}

\begin{proof}
Suppose that  the scalar curvature of \( X \) satisfies \( Sc(X) >C_{1,n}\cdot\sigma_X \).
We define a constant $C_{2,n}$ as $2^{n/2 + 1}$. It is determined in equation (1). Suppose, by contradiction, that 
\[
\mathrm{stsys}_2(X, \partial X) \leq \frac{C_{2,n}}{C_{1,n}\cdot\sigma_X},
\]
Fix any sufficiently small $\epsilon > 0$. Since $K$ is a scalably contractible complex, there exists a finite covering 
\[
p : \tilde{K}_\epsilon \to K
\]
and an \( \epsilon \)-Lipschitz map
\[
f_\epsilon : \tilde{K}_\epsilon \to S^k
\]
such that $\displaystyle\sum_{e^n_\lambda}[f(e^n_\lambda):S^k]$ is non-zero where $e^n_\lambda$ are top-dimensional cells of $X$ and $[f(e^n_\lambda):S^k]$ is the attaching coefficient.

By considering a distance sphere in $K_\epsilon$, there exists a finite covering $\pi: \tilde{X}_\epsilon\to X$ of $X$ and a map 
\[
\tilde{r}: \tilde{X}_\epsilon \to \tilde{K}_\epsilon
\]
lifting the retraction \( r : X \to K \). The Lipschitz constant of \( \tilde{r} \) is bounded above by a constant \( c(r) \) depending only on \( r \), and independent of \( \epsilon \).

Now consider the following composition:
\[
f': 
\tilde{X}_\epsilon 
\xrightarrow{\tilde{r} \times \varphi} \tilde{K}_\epsilon \times S^{n-k-2} \times D^2 
\xrightarrow{f_\epsilon \times \mathrm{id}} S^k \times S^{n-k-2} \times D^2 
\xrightarrow{\wedge \times \mathrm{id}} S^{n-2} \times D^2.
\]
where $\varphi$ is the minimizer which attains $L_{n-k}(X)$ (precisely we take a minimizing sequence of maps). 

Next, we fix an embedding
\[
\iota: (S^{n-2} \times D^2,\; S^{n-2} \times S^1) \to (S^n, \Sigma),
\]
where \( \Sigma \subset S^n \) is a loop. We assume that \( \iota \) is chosen to minimize the Lipschitz constant among all such embeddings and all choices of loops \( \Sigma \). We define $F'=\iota\circ f'$.

For each point \( x \in \tilde{X}_\epsilon \), the differential
\[
dF'_x : T_x \tilde{X}_\epsilon \to T_{F'(x)} S^n
\]
is a linear map between finite-dimensional inner product spaces. By the singular value decomposition, there exist orthonormal bases
\[
\{ a_1, \dots, a_n \} \subset T_x \tilde{X}_\epsilon, \quad \{ b_1, \dots, b_n \} \subset T_{F'(x)} S^n
\]
such that
\[
dF'_x(a_i) = \mu_i b_i \quad \text{for } i = 1, \dots, n,
\]
where \( \mu_1 \geq \mu_2 \geq \cdots \geq \mu_n \geq 0 \) are the singular values of \( dF'_p \).

Since \( f_\epsilon \) satisfies \( \operatorname{Dil}_1(f_\epsilon) \leq \epsilon \), we have
\[
L_{n-k}(X) \geq \mu_1 \geq \mu_2,
\]
and
\[
\sum_{2 < j \leq n} \mu_j = O(\epsilon).
\]

Therefore, the exterior square of the differential,
\[
\wedge^2 dF'_x : \wedge^2 T_x \tilde{X}_\epsilon \to \wedge^2 T_{F'(x)} S^n,
\]
has the principal singular value \( \mu_1 \mu_2 \leq L_{n-k}(X)^2 \), and all other singular values satisfy
\[
\mu_{j_1} \mu_{j_2} = O(\epsilon),
\quad \text{for } 2 < j_1 < j_2 \leq n.
\]

We consider the spinor bundle \( E_0 \) over the unit sphere \( S^n \) and the spinor bundle $\slashed{S}$ over $\tilde{X}_\epsilon$.  
Define the pullback bundle induced by the map \( F': \tilde{X}_\epsilon \to S^n \)
\[
E := (F')^* E_0.
\]

Let \( R_E \) denote the curvature term appearing in the Weitzenb\"ock formula for the square of the twisted Dirac operator on \( \slashed{S} \otimes E \), so that for any section \( s \in \Gamma(\slashed{S} \otimes E) \),
\[
D^2_{\slashed{S} \otimes E} s
= (\nabla^{\slashed{S} \otimes E})^* \nabla^{\slashed{S} \otimes E} s + \frac{1}{4} Sc(g) s + R_E s.
\]

The following estimate for the curvature term \( R_E \) is well known:
\[
|R_E s| \leq \frac{1}{4} \sum_{\substack{1 \leq j, k \leq n \\ j \neq k}} \mu_j \mu_k \, |s|,
\]
for all \( s \in \slashed{S}_x \otimes E_x \). Thus $|R_E|_\infty \leq C_{1,n}\cdot\frac{\mu_1\mu_2}{2}$ for sufficiently small $\epsilon >0$. ($C_{1,n}$ is determined by $\wedge$ and $\iota$.)

Next we construct the hermitian vector bundle $F$ relative to $E$ so that the curvature of $F$ satisfies $||R_F||\leq C_{1,n}\cdot L_{n-k}(X)^2/2$. 
%pushforward line bundle copy comass mass duality

We reduce the problem to constructing a differential 1-form \( \omega \) such that 
\[
\omega = d\theta \quad \text{near } \partial X,
\]
and the 2-form \( d\omega \) satisfies the comass norm estimate
\[
\|d\omega\| \leq C_{1,n}\cdot2^{-n/2 - 1} \cdot L_{n-k}(X)^2.
\]
(Here, we consider the deformation of the map $F'$ along the loop as in \cite{orikasa2025analysiscontractionmappingscomplement} Theorem 3.7. and use the splitting of $E$ and reduce to the problem on the base space, since the problem is local.)

Let us consider a 2-cycle 
\[
[D^2] \in H_2(X, \partial X; \mathbb{R}),
\]
and define a class 
\[
u \in H^2_{\mathrm{cpt}}(\mathring{X}; \mathbb{R})
\]
to be the dual of \([D^2]\) under the Kronecker pairing. 

We describe a de Rham representative of $u$ in \( H^2_{\mathrm{cpt}}(\mathring{X}; \mathbb{R}) \). Since each fiber of \( X\) is isomorphic to \( D^2\), we may introduce polar coordinates \( (r, \theta) \) on each fiber, where \( r > 0 \) and \( \theta \in S^1 \). Then, we define the \emph{angle form} as
\[
\frac{1}{2\pi} d\theta,
\]
which is a closed 1-form along the fiber direction and integrates to 1 over each circle fiber.

To construct a compactly supported 2-form representing $u$, let \( r: X\setminus \{0\} \to \mathbb{R}_{>0} \) be the radial coordinate function. 

Let \( \rho(r) \colon [0, \infty) \to [0, 1] \) be a smooth function such that:
\[
\rho(r) =
\begin{cases}
0 & \text{for } r < \varepsilon, \\
1 & \text{for } r > 2\varepsilon,
\end{cases}
\]
with a smooth transition in between. 

Then a representative of \( u \) can be given by the differential form
\[
u = \left[ d\left( \rho(r) \cdot \frac{1}{2\pi} d\theta \right) \right],
\]
where \( \rho(r) \) is a smooth cutoff function equal to 1 near \( \partial X \) and supported away from the interior.

Since 
\[
\operatorname{stsys}_2(X,\partial X) = \|[D^2]\| > \frac{C_{2,n}}{C_{1,n}\cdot\sigma_X},
\]
by duality we obtain
\begin{equation}\label{eq:ref}
    \|u\|^* < \frac{C_{1,n}\cdot\sigma_X}{ C_{2, n}}.
\end{equation}

(Therefore, we choose $C_{2,n}$ as $2^{n/2 + 1}$.)

Hence, there exists a compactly supported 1-form \( \eta \in \Omega^1_{\mathrm{cpt}}(\mathring{X}) \) such that
\[
\left\| d\left( \rho(r) \cdot \frac{1}{2\pi} d\theta + \eta \right) \right\| < \frac{C_{1,n}\cdot\sigma_X}{ C_{2,n}}.
\]
Since \( C_{2,n} = 2^{n/2 + 1} \), we obtain
\[
\left\| d\left( \rho(r) \cdot \frac{1}{2\pi} d\theta + \eta \right) \right\| < 2^{-n/2 - 1} \cdot C_{1,n}\cdot\sigma_X.
\]

Now we define the 1-form
\[
\omega = \rho(r) \cdot \frac{1}{2\pi} d\theta + \eta.
\]
This form satisfies \( \omega = d\theta \) near \( \partial X \), since \( \rho(r) \equiv 1 \) there and \( \eta \) is supported away from the boundary. Therefore, \( \omega \) satisfies both the desired boundary behavior and the comass norm estimate.

Since \( g \) is an adaptable metric, there exists a complete extension \( \bar{g} \) of \( g \) over the end-completed space
\[
X'_\epsilon := X_\epsilon \cup (\partial X_\epsilon \times [0, +\infty)),
\]
and there exists a relative pair \( (\pi_* E, F) \) of vector bundles over \( X'_\epsilon \). By the relative index theorem, we have
\[
\operatorname{ind}(\slashed{D}_E) - \operatorname{ind}(\slashed{D}_F) = \int_{X'_\epsilon} \hat{A}(T X'_\epsilon) \wedge \operatorname{ch}(\pi_* E, F).
\]

On the other hand, from vanishing results, we know that
\[
\operatorname{ind}(\slashed{D}_E) = \operatorname{ind}(\slashed{D}_F) = 0
\quad \text{and} \quad 
\int_{X'_\epsilon} \hat{A}(T X'_\epsilon) \wedge \operatorname{\hat{ch}}(F) = 0,
\]
by the construction of \( F \) and use $u^2=0$ as a cohomology class where $\operatorname{\hat{ch}}(F)$ is the reduced chern character of $F$.

However, we compute
\begin{align*}
\int_{X'_\epsilon} \hat{A}(T X'_\epsilon) \wedge \operatorname{\hat{ch}}(\pi_* E)
&= \int_{X_\epsilon} \hat{A}(T X_\epsilon) \wedge \operatorname{\hat{ch}}(\pi_* E) \\
&= \int_{\widetilde{X}_\epsilon} \hat{A}(T \widetilde{X}_\epsilon) \wedge \operatorname{\hat{ch}}(E) \\
&= \int_{\widetilde{X}_\epsilon} \hat{A}(T \widetilde{X}_\epsilon) \wedge (F')^* \operatorname{\hat{ch}}(E_0) \\
&= \deg(F') \cdot \int_{S^n} \operatorname{\hat{ch}}(E_0) \\
&= \deg(F') \neq 0.
\end{align*}

This contradicts the vanishing of the relative index, completing the proof.
\end{proof}

\bibliographystyle{amsalpha}
\bibliography{reference3.bib}
\end{document}